\documentclass{amsart}

\usepackage{amssymb}
\usepackage{amsmath}
\usepackage{bbm}
\usepackage{amscd}
\usepackage{amsfonts}
\usepackage{amsthm}
\usepackage{mathrsfs}
\usepackage{verbatim}
\usepackage[colorlinks]{hyperref}
\usepackage{fullpage}
\usepackage{mathdots}
\usepackage{graphicx,subfigure}
\usepackage[english]{babel}
\usepackage{tikz}
\usetikzlibrary{arrows}

\usetikzlibrary{backgrounds,fit,decorations.pathreplacing}
\usetikzlibrary{positioning}
\usetikzlibrary{calc,through,chains}
\usetikzlibrary{arrows,shapes,snakes,automata, petri}

\theoremstyle{definition}
\newtheorem{Theorem}[equation]{Theorem}
\newtheorem{Corollary}[equation]{Corollary}
\newtheorem{Lemma}[equation]{Lemma}

\newtheorem{Remark}{Remark}

\newtheorem{Definition}[equation]{Definition}

\numberwithin{equation}{section}
\numberwithin{figure}{section}

\newcommand{\PP}{{\mathbb P}}
\newcommand{\C}{{\mathbb C}}
\newcommand{\Z}{{\mathbb Z}}
\newcommand{\Q}{{\mathbb Q}}
\newcommand{\R}{{\mathbb R}}

\newcommand{\mt}[1]{\text{#1}}

\begin{document}
\date{March 15, 2019}

\title{The Betti Numbers of a Determinantal Variety}

\author{Mahir Bilen Can}
 \address{Department of Mathematics\\
  Tulane University\\
  New Orleans, Louisiana 70118}
\email{mahirbilencan@gmail.com}

\begin{abstract}
We determine the Poincar\'e polynomial of the determinantal variety $\{\det = 0\}$ in 
the projective space associated with the monoid of $n\times n$ matrices.  
\vspace{.5cm}

\end{abstract}

\maketitle

\section{Introduction}\label{S:1}

In this note, we look closely at the homology groups of a classical variety.
Let $Y^0$ denote the semigroup defined by the vanishing of the determinant polynomial in $n\times n$ matrices.
More precisely, we set $Y^0:=M\setminus G$, where $M$ is the monoid of all linear operators on an $n$ dimensional
complex vector space $V$, and $G= \textrm{GL}(V)$. 
The purpose of this note is to describe the Poincar\'e polynomial of the quotient, 
\begin{align}\label{A:asin}
Y:=\PP(Y^0) = (Y^0 \setminus \{ 0 \}) /\C^*,
\end{align}
where $\C^*$ is the center of $\textrm{GL}(V)$.

It is not difficult to see that if $\dim V = 2$, then $Y$ is isomorphic to the quadric surface $\PP^1\times \PP^1$ in $\PP^3$.
In particular, the Poincar\'e polynomial of $Y$ is given by $1 + 2t^2 + t^4$. 
However, in general, $Y$ has a large singular locus, which is given by
the projectivization of a $G\times G$-orbit closure, \hbox{$(\overline{GeG}\setminus \{0\})/\C^*$},
where $e$ is an idempotent of rank $n-2$ in $M$. 
It is natural question to ask for a description of the Poincar\'e polynomial of $Y$ for $n:=\dim V >2$. 
It turns out that the degrees as well as the coefficients of monomials in $P_Y(t)$ have interesting patterns, 
although $P_Y(t)$ is neither symmetric nor unimodal.

Our main result and its corollary are the following statements.

\begin{Theorem}\label{T:1}
Let $Y$ denote the determinantal variety defined as in (\ref{A:asin}).
If $V$ is $n$ dimensional, then the homology groups of $Y$ satisfy the following isomorphisms:
\[
H_k (Y) \cong 
\begin{cases}
0 & \text{ if $k$ is odd and $k < n^2-1$}; \\
\Z & \text{ if $k $ is even and $k  < n^2-1$};\\
H_{k+1- (n^2-1)}(\textrm{PSU}_n) & \text{ if $k$ is odd and $k \geq n^2-1$}.
\end{cases}
\]
Finally, if $k$ is even and $k \geq n^2-1$, then we have 
\hbox{$H_k (Y) / H_{k+1-(n^2-1)}(\textrm{PSU}_n)\cong \Z$.}
Here, $\textrm{PSU}_n$ denotes the projective special unitary group.
\end{Theorem}

Let us denote by $P_{\textrm{PSU}_n}(t)$ the polynomial $\prod_{i=1}^{n-1} (1+ t^{2i+1})$. 
In other words, $P_{\textrm{PSU}_n}(t)$ is the Poincar\'e polynomial of $\textrm{PSU}_n$.
It is easy to check that, starting from $n=5$ the polynomial $P_{\textrm{PSU}_n}(t)$ is no longer unimodal. 
On the other hand, as a product of palindromic polynomials,
$P_{\textrm{PSU}_n}(t)$ is palindromic for all $n$. 
Let us write $P_{\textrm{PSU}_n}(t)$ in the form $P_{\textrm{PSU}_n}(t) = \sum_{i=0}^{n^2-1} b_i t^i$ with $b_i \in \Z_{\geq 0}$.

\begin{Corollary}\label{C:1}
Let $Y$ denote the determinantal variety defined as in (\ref{A:asin}).
If $V$ is $n$ dimensional, then the Poincar\'e polynomial of $Y$ is expressible as a sum of two polynomials, 
\begin{align}\label{I:Poincare}
P_Y(t) &=  A(t)+\widetilde{B}(t),
\end{align}
where $A(t) = 1+ t^2+ \cdots + t^{2\lfloor \frac{n^2-1}{2} \rfloor}$, and $\widetilde{B}(t)$ is the polynomial that is obtained from 
\[
B(t) := t^{n^2-1}P_{\textrm{PSU}_n}(t) = \sum_{i=n^2-1}^{2(n^2-1)} b_{i-(n^2-1)} t^{i}
\]
by adding 1 to the coefficients of the terms $b_{i-(n^2-1)} t^{i}$ with $i$ odd.
\end{Corollary}

Note that a complete description of the (torsion in the) cohomology ring of $\textrm{PSU}_n$ has recently been given 
by Haibao Duan in~\cite{Duan}.

\section{Preliminaries}\label{S:2}

We start with reviewing some well known facts about the Chow groups and Borel-Moore homology
groups. We follow the presentation in~\cite[Chapter 19]{Fulton}; if $X$ is a topological space, 
then $\bar{H}_*(X)$ stands for the Borel-Moore homology group with integer coefficients.

\subsection{}

Let $k$ be a nonnegative integer, and let $X$ be a scheme. 
The free abelian group generated by all $k$ dimensional 
subvarieties of $X$ is denoted by $Z_k X$. The elements of $Z_k X$
are called $k$-cycles. 
A $k$-cycle $\alpha$ is said to be rationally equivalent to 0,
and written $\alpha \sim 0$ if there are a finite number of 
$k+1$ dimensional subvarieties $Y_1,\dots, Y_s$ and 
rational functions 
\hbox{
$f_i\in \C(Y_i)^*$ ($i=1,\dots, s$)}
such that 
$\alpha = \sum [\mt{div}(f_i)]$.
The set of $k$-cycles which are rationally equivalent to 0 is
a subgroup of $Z_k(X)$, denoted by $Rat_k(X)$.
The quotient group 
$
A_k(X):= Z_k(X)/ Rat_k(X) 
$
is called the {\em group of $k$-cycles modulo rational equivalence}, or 
the $k$-th {\em Chow group}.
The total Chow group $A_*(X):= \bigoplus_{k=0}^{\dim X} A_k (X)$ is a graded 
abelian group; if $X$ is equidimensional, then $A_{\dim X}(X)$ is freely generated by 
the classes of irreducible components of $X$.

If $X$ is an equidimensional scheme, by replacing $Z_k(X)$ 
with $Z^k(X)$, that is the group of $k$-codimensional cycles,
we have the Chow group
$$
A^k(X) := Z^k(X)/ Rat_{\dim X -k} (X) = A_{\dim X -k}(X).
$$ 
We set $A^*(X):= \oplus A^i(X)$.
If $X$ is smooth, then there is an intersection pairing on $A^*(X)$, 
and hence, $A^*(X)$ becomes a ring.

Let $s$ be an element of $\Z_{\geq 0} \cup \{ * \}$. 
We will denote the vector spaces $A^s(X)\otimes \Q$ 
and $A_s(X)\otimes \Q$ by $A^s(X)_ \Q$ and $A_s(X)_ \Q$, respectively.

%\subsection{Some useful properties of Chow groups.}

Chow groups behave nicely with respect to certain classes of morphisms. 
\begin{enumerate}
\item If $f : X\rightarrow Y$ is a proper morphism, 
then there is a (covariant) homomorphism 
$$
f_*: A_k (X) \rightarrow A_k(Y).
$$
\item If $f: X \rightarrow Y$ is flat morphism of relative dimension $n$, 
then there is a (contravariant) homomorphism 
$$
f^*: A_k (Y) \rightarrow A_{k+n} (X).
$$
\end{enumerate}

Let $i: Y\hookrightarrow X$ be an inclusion of a closed subscheme $Y$
into a scheme $X$. Let $U$ denote the complement $X - Y$ 
and let $j : U \rightarrow X$ denote the inclusion. 
Then there is an exact sequence
\begin{align}\label{A:inclusion of U 0}
A_k (Y) \stackrel{i_*} {\longrightarrow} A_k (X) \stackrel{j^*}{\longrightarrow} A_k(U) \rightarrow 0
\end{align}
for all $k$. To understand the image of $i_*$ in $A_k(X)$
we need to consider Edidin and Graham's version of
Bloch's higher Chow groups.

%\subsubsection{}

Let $X$ be a quasi-projective scheme, and 
let $\Delta^k$ denote the algebraic version of the regular $k$-simplex:
$$
\Delta^k = \mt{Spec} ( \Z [t_1,\dots, t_k ] / (t_1+\cdots + t_k -1)).
$$

A {\em face} of $X\times \Delta^k$ is the subscheme of the form 
$X\times \Delta^m$, where the second factor $\Delta^m$ is the image
of an injective canonical morphism $\rho : \Delta^m \rightarrow \Delta^k$. 
We denote by $Z^i(X,\bullet)$ the complex
whose $k$-th term is the group of cycles of codimension 
$i$ in $X\times \Delta^k$ which intersect properly all of the faces 
in $X\times \Delta^k$.
In~\cite{Bloch86}, Bloch considered the following higher Chow groups:
$$
CH^i(X,m) := H_m ( Z^i(X, \bullet )).
$$

Let $Z_p(X,\bullet)$ denote 
the complex
whose $k$-th term is the group of cycles of dimension 
$p+k$ in $X\times \Delta^k$ intersecting the faces properly.
\begin{Definition}
The {\em $(p,k)$-th higher Chow group}
of a quasi-projective scheme $X$ is defined by 
\begin{align}
A_p ( X, k) := H_k ( Z_p( X,\bullet)).
\end{align}
\end{Definition}
The point of this definition is that $X$ does not need to be equidimensional. 
If $X$ is equidimensional of dimension $n$, then it is easy to see that 
$A_p ( X, k) = CH^{n-p}(X,k)$.

Now we state the localization long exact sequence for 
higher Chow groups. 
\begin{Lemma}\label{L:BEG long exact}
Let $Y$ be a closed, not necessarily equidimensional subscheme of 
an equidimensional scheme $X$.
Then there is a long exact sequence of higher Chow groups;
\begin{align}\label{A:BEG long}
\cdots \rightarrow  A_p(Y,k)  \rightarrow & A_p(X,k) \rightarrow 
A_p(X-Y,k) \rightarrow \cdots  \\
\cdots \rightarrow A_p(X-Y,1)  \rightarrow &A_p(Y) \rightarrow A_p(X) \rightarrow 
A_p(X-Y) \rightarrow 0.
\end{align}
\end{Lemma}
\begin{proof}
See~\cite[Lemma 4]{EdidinGraham}.
\end{proof}

\begin{Remark}
It is not clear if the localization long exact sequence terminates for an arbitrary scheme. 
\end{Remark}

\subsection{}

The Borel-Moore homology groups of a space are defined by using 
ordinary cohomology groups as follows. Let $Y$ be a topological space that is embedded as a 
closed subspace of $\R^n$ for some positive integer $n$. Then the $q$th Borel-Moore homology
of $Y$, denoted by $\bar{H}_q (Y)$ is defined by 
\[
\bar{H}_q (Y) = H^{n-q} ( \R^n, \R^n \setminus Y).
\]

\begin{enumerate}
\item If $f: Y\rightarrow X$ is a proper morphism of complex schemes, 
then there are covariant homomorphisms $f_*: \bar{H}_i(Y) \rightarrow \bar{H}_i(X)$.

\item If $j : U\hookrightarrow Y$ is an open imbedding, then 
there are contravariant restriction homomorphisms $j^*: \bar{H}_i(Y) \rightarrow \bar{H}_i(U)$.

\item If $Y$ is the complement of $U$ in $X$ and $i: Y \rightarrow X$ is the 
closed imbedding, then there is a long exact sequence 
\begin{align}\label{A:LES}
\cdots \rightarrow \bar{H}_{i+1}(U) 
\rightarrow \bar{H}_{i}(Y)
\overset {i_*}{\longrightarrow}
\bar{H}_{i}(X) 
\overset {j^*}{\longrightarrow}
\bar{H}_{i}(U) 
\rightarrow \bar{H}_{i-1}(Y) 
\rightarrow \cdots 
\end{align}

\item If $X$ is a disjoint union of a finite number of spaces, $X= X_1 \cup \cdots \cup X_n$,
then $\bar{H}_i(X) = \oplus \bar{H}_i (X_j)$.

\item There is a K\"unneth formula for Borel-Moore homology.

\item If $X$ is an $n$-dimensional complex scheme, 
then $\bar{H}_i (X) = 0$ for all $i>2n$, and $\bar{H}_{2n} (X)$ is a free abelian group
with one generator for each irreducible component $X_i$ of $X$. 
The generator corresponding to $X_i$ will be denoted by $cl(X_i)$. 
More generally, we will use the following notation: If $V$ is a 
$k$-dimensional closed subscheme 
of $X$, and $i: V \hookrightarrow X$ is the closed imbedding, then 
$cl_X(V)$ stands for $i_* cl(V)$, which lives in $\bar{H}_{2k}(X)$.
If confusion is unlikely, we will omit the subscript $X$ from the notation.

\item If $f: V\rightarrow W$ is a proper, surjective morphism of varieties, then 
$f_* cl(V) = \mt{deg}(V/W) \cdot cl(W)$. Since we do not need it for our purposes,  
we will not define $\mt{deg}(V/W)$ here; see~\cite[Section 1.4]{Fulton} for its definition.

\item For any complex scheme $X$, 
there is a homomorphism from algebraic $k$-cycles on $X$ 
to the $k$-th Borel-Moore homology, 
$cl : Z_k (X) \rightarrow \bar{H}_{2k}(X)$, defined 
by $cl( \sum n_i [V_i] ) = \sum n_i cl_X(V_i)$. 
This homomorphism factors through the ``algebraic equivalence''
(which we didn't introduce), hence, by composition, 
it induces a homomorphism from the $k$-th Chow group 
of $X$ onto the $2k$-th Borel-Moore homology.
We will denote the resulting homomorphism by $cl$ also,
and call it the {\em cycle map}.

\item If a complex scheme $X$ has a cellular decomposition, then the cycle map 
$cl_X : A_k(X) \rightarrow \bar{H}_{2k}(X)$
is an isomorphism (see~\cite[Section 19.1.11]{Fulton}).

\item Finally, let us mention that if $X$ is an $n$-dimensional
oriented manifold, then $\bar{H}_k(X) \cong H^{n-k}(X)$. 

\end{enumerate}

For further details of this useful homology theory, see~\cite{BorelMoore}.

\section{Proof}\label{S:3}

We will use the following notation in the sequel:
\[
\begin{array}{lcl}
M &: & \text{ the monoid of $n\times n$ matrices defined over $\C$;}\\
G &: & \text{ the general linear group of $n\times n$ matrices defined over $\C$;}\\
T &: & \text{ the maximal torus consisting of diagonal matrices in $G$;}\\
Z &: & \text{ the center of $G$;}\\
X &: & \text{ the projectivization of $M$, $X:=(M\setminus \{0\}) / Z$;}\\
Y_0 & : & \text{ the vanishing locus of the determinant in $M$;}\\
Y &:& \text{ the projectivization of $Y$, $Y:=Y^0/Z$;}\\
U & :& \text{ the projectivization of $G$, $U:=G/Z = \textrm{PGL}_n$.}\\
\end{array}
\]

Since $Y$ is a projective variety, we have $\bar{H}_q(Y) = H_q(Y)$ for $q\in \{ 0,\dots, \dim Y\}$.
Of course, similar equalities hold true for $X\cong \PP^{n^2-1}$ as well.
Both of the spaces $X$ and $Y$ are path connected, therefore, we have $H_0(X) = H_0(Y) = \Z$.
The complement of $Y$ in $X$ is given by the group $U$. 
Since $U$ is open in $X$, there is a long exact sequence for their Borel-Moore homology,
\begin{align}\label{A:longexact}
\dots \to \bar{H}_q (Y) \to \bar{H}_q(X) \to \bar{H}_q (U) \to 
\bar{H}_{q-1}(Y) \to \dots,
\end{align}
As complex projective spaces have zero odd homology, 
the long exact sequence in (\ref{A:longexact}) breaks up into short exact sequences.
More precisely, for $q=1,\dots, n^2-1$, we have 
\begin{align}\label{A:short exact}
0 \to \bar{H}_{2q+1} (U) \to {H}_{2q}(Y) \to H_{2q} (X) \to 
\bar{H}_{2q}(U) \to H_{2q-1}(Y) \to 0.
\end{align}

We will identify $U=\textrm{PGL}_n$ with the (complex) projective special linear group, $ \textrm{PSL}_n$. 
In turn, as a real manifold, $ \textrm{PSL}_n$ has the (Cartan-Malcev-Iwasawa) decomposition
$\textrm{PSL}_n \cong \textrm{PSU}_n \times \R^s$, where $\textrm{PSU}_n$ is the projective 
special unitary group, and $s= n^2-1$. 
Note that, as a (real) Lie group, $\textrm{PGL}_n$ is an oriented $2(n^2-1)$-dimensional manifold,
therefore, its Borel-Moore homology groups are actually cohomology groups, 
\begin{align*}
\bar{H}_q ( U) = H^{2(n^2-1)-q}(U)
= H^{2(n^2-1)-q}(\textrm{PSL}_n) = H^{2(n^2-1)-q}(\textrm{PSU}_n). 
\end{align*}
The unitary groups are compact. 
Since $\textrm{PSU}_n$ is a $(n^2-1)$-manifold, by Poincar\'e duality, we see the following fact.

\begin{Lemma}
The homology groups of $U=\textrm{PGL}_n$ are given by 
\begin{align}\label{A:PSU}
\bar{H}_q ( U) = 
\begin{cases}
0 & \text{ if } q < n^2-1, \\
H_{q-(n^2-1)}(\textrm{PSU}_n) & \text{ if } q \geq n^2-1.
\end{cases} 
\end{align}
\end{Lemma}

By using (\ref{A:PSU}) and the short exact sequence in (\ref{A:short exact}), 
we determine the homology groups of $Y$ in lower degrees. 

\begin{Lemma}\label{L:lowerdegree}
The homology groups $H_q(Y)$ for $q<n^2-1$ are given by  
\begin{align}\label{A:lowerdata}
H_q (Y) = 
\begin{cases}
0 & \text{ if $q$ is odd and $q< n^2-1$}, \\
\Z & \text{ if $q$ is even and $q < n^2-1$}.
\end{cases}
\end{align}
\end{Lemma}

\begin{Remark}
Since $Y$ is an irreducible hypersurface in $X$, 
the knowledge of the lower degree homology groups as in (\ref{A:lowerdata}) 
could also be obtained by using the Lefschetz hyperplane theorem,
see~\cite[Corollary 1.24]{Voisin}. 
\end{Remark}

We are now ready to state and prove our main result that is stated in the introduction. 

\begin{comment}

\begin{Theorem}\label{T:main theorem}
Let $Y$ denote the projectivization of the vanishing locus of the determinant polynomial in $\textrm{Mat}_n$. 
Then the homology groups of $Y$ satisfy the following isomorphisms:
\[
H_k (Y) \cong 
\begin{cases}
0 & \text{ if $k$ is odd and $k < n^2-1$}; \\
\Z & \text{ if $k $ is even and $k  < n^2-1$};\\
H_{k+1- (n^2-1)}(\textrm{PSU}_n) & \text{ if $k$ is odd and $k \geq n^2-1$}.
\end{cases}
\]
Finally, if $k$ is even and $k \geq n^2-1$, then we have 
\[
H_k (Y) / H_{k+1-(n^2-1)}(\textrm{PSU}_n)\cong \Z. 
\]
\end{Theorem}
\end{comment}

\begin{proof}[Proof of Theorem~\ref{T:1}]
For $q\in \{1,\dots, n^2-2\}$, we have the commuting diagram of Chow groups and Betti numbers as in Figure~\ref{F:commuting}. 
\begin{figure}[htp]
\begin{center}
\begin{tikzpicture}{scale = .35}
\node at (-1.25,1) (a) {$A_q(X)$};
\node at (-4.25,1) (a') {$A_q(Y)$};

\node at (1.75,1) (b) {$A_q(U)$};
\node at (3.75,1) (b') {$0$};
\node at (-1.25,-1) (c) {$\bar{H}_{2q}(X)$};
\node at (-4.25,-1) (c') {$\bar{H}_{2q}(Y)$};
\node at (-7,-1) (e') {$\bar{H}_{2q+1}(U)$};
\node at (-7,1) (e) {$A_{q+1}(U,1)$};
\node at (1.75,-1) (d) {$\bar{H}_{2q}(U)$}; 
\node at (4.75,-1) (d') {$\bar{H}_{2q-1}(Y)$};
\node at (-3.75,0) {$cl_Y$}; 
\node at (-.75,0) {$cl_X$}; 
\node at (2.25,0) {$cl_U$}; 
\node at (0.25,1.35) {$i^*$}; 
\node at (0.25,-.65) {$i^*$}; 
\node at (-2.75,1.35) {$j_*$}; 
\node at (-2.75,-.65) {$j_*$}; 
\node at (3,-.65) {$\delta_{2q}$}; 
\node at (-5.5,-.65) {$\delta_{2q+1}$}; 
\node at (7,-1) (g) {$0$};
\node at (-9,-1) (g') {$0$};

\draw[->,thick] (d') to (g);
\draw[->,thick] (g') to (e');

\draw[->,thick] (a) to (b);
\draw[->,thick] (a') to (a);
\draw[->,thick] (a') to (c');
\draw[->,thick] (c') to (c);
\draw[->,thick] (a) to (c);
\draw[->,thick] (b) to (d);
\draw[->,thick] (b) to (b');
\draw[->,thick] (d) to (d');
\draw[->,thick] (c) to (d);
\draw[->,thick] (e') to (c');

\draw[->,thick] (e) to (a');
\end{tikzpicture}
\end{center}
\label{F:commuting}
\caption{Breaking of the long exact sequences.}
\end{figure}

We have two remarks in order:
\begin{enumerate}
\item 
Since $X$ has a cellular decomposition, the vertical map $cl_X$ is an isomorphism. 

\item Secondly, as a result of a deep result Totaro, we know that the Chow groups of $U$ are almost always zero,
except at the degree $n^2-1$, where it is $\Z$. Indeed, by~\cite[Theorem 16.6]{Totaro}, we know that the  
Chow ring $A^*(\textrm{GL}_n/ \C^*)$, which is Poincar\'e dual to $A_*(U)$, is isomorphic to $\Z$.
In particular, $A_q (U) = 0$ for $q\in \{0,\dots, n^2-3\}$, and $A_{0}(U)\cong \Z$.
\end{enumerate}

As a consequence of these two remarks, we see that, for $q\in \{1,\dots, n^2-3\}$, 
the map $i^*$ in the top row of diagram in (\ref{F:commuting}) is zero, 
hence, the top $j^*$ is surjective.  It follows that the bottom $j^*$ is surjective as well. 
But then, by the exactness of the bottom row, 
the kernel of the bottom $i^*$ is equal to $\bar{H}_{2q}(X)$, hence it is the zero map. 
In other words, we have 
\[
\bar{H}_{2q}(Y)/ \bar{H}_{2q+1}(U) \cong \Z \ \text{ and }\ \bar{H}_{2q-1}(Y) \cong \bar{H}_{2q}(U).
\]
Thus, combining these isomorphisms with Lemma~\ref{L:lowerdegree}, we finish the proof of our theorem.

\end{proof}

It is now easy to verify that the Poincar\'e polynomial of $Y$ is as given in Corollary~\ref{C:1}.

\bibliographystyle{plain}
\bibliography{referenc}

\end{document}